\documentclass[final]{amsart}

\usepackage{cite}

\usepackage[english]{babel}

\usepackage{cite}


\usepackage{amsmath,amsthm,amsfonts,amssymb,amstext,euscript,verbatim,bbm}
\usepackage{mathrsfs}
\usepackage{hyperref}
\usepackage{doi,cite}

\usepackage{enumerate}
\usepackage{graphicx}
\usepackage{caption}
\usepackage{amscd}

\usepackage[color]{showkeys}

\newtheorem{theorem}{Theorem}

\newtheorem{lemma}[theorem]{Lemma}
\newtheorem{proposition}[theorem]{Proposition}

\theoremstyle{definition}
\newtheorem{remark}{Remark}
\newtheorem{example}{Example}

\usepackage{marginnote}
\usepackage{color}

\newcommand{\bfx}{\mathbf{x}}

\newcommand{\bfy}{\mathbf{y}}

\newcommand{\bfz}{\mathbf{z}}
\newcommand{\bff}{\mathbf{f}}

\newcommand{\bfF}{\mathbf{F}}

\numberwithin{equation}{section}
\numberwithin{theorem}{section}

\allowdisplaybreaks

\let\svthefootnote\thefootnote

\begin{document}

\title[Averaging with a time-dependent perturbation]{Averaging with a time-dependent perturbation parameter}
\author{\textsc{David Fajman}}
\author{\textsc{Gernot Hei}\ss\textsc{el}}
\address[\textsc{David Fajman} and \textsc{Gernot Hei}\ss\textsc{el}]{University of Vienna, Department of Physics, Strudlhofgasse 4, 1090 Vienna, Austria}
\email{david.fajman@univie.ac.at}
\email{gernot.heissel@univie.ac.at}
\author{\textsc{Jin Woo Jang}}
\address[\textsc{Jin Woo Jang}]{University of Bonn, Institute for Applied Mathematics, Endenicher Allee 60, Bonn, Germany}
\email{jangjinw@iam.uni-bonn.de}

\begin{abstract}
Motivated by recent problems in mathematical cosmology, in which temporal averaging methods are applied in order to analyse the future asymptotics of models which exhibit oscillatory behaviour, we provide a theorem concerning the large-time behaviour for solutions of a general class of systems. We thus propose our result to be applicable to a wide range of problems in spatially homogenous cosmology with oscillatory behaviour. Mathematically the theorem builds up on the standard theory of averaging in non-linear dynamical systems.
\end{abstract}
 
\let\thefootnote\relax

\thanks{The authors would like to thank Piotr Bizo\'n for discussions. The first and second authors acknowledge support of the Austrian Science Fund (FWF) through the Project \emph{Geometric transport equations and the non-vacuum Einstein flow} (P 29900-N27). The second and third authors gratefully acknowledge the support of the Hausdorff Research Institute for Mathematics
(Bonn), through the Junior Trimester Program on Kinetic Theory. The third author was supported by CRC 1060 \textit{The mathematics of emergent effects} at the
University of Bonn funded through the German Science Foundation (DFG)}

\newcommand{\eqdef }{\overset{\mbox{\tiny{def}}}{=}}

\addtocounter{footnote}{-1}\let\thefootnote\svthefootnote

\thispagestyle{empty}

\maketitle

\section{Introduction}\label{S: intro}

\subsection{Averaging}

The theory of averaging in nonlinear dynamical systems has its origins in the 18th century, when perturbative methods for differential equations became important to test the predictions of Newtons theory of gravitation for celestial mechanics against observational data. The motion of a planet is described by a Keplerian ellipse only in good approximation, since it represents the solution to the two-body problem planet-sun. The actual trajectory of a planet is however also influenced by all other celestial bodies such as their own satelites. The influence of these bodies results in (generally periodic) perturbations from Keplerian motion. The necessity to take into account these effects due to the advances in accuracy of astronomical observations lead to the development of perturbative methods for differential equations, and eventually to the birth of averaging methods in nonlinear dynamical systems; cf.~\cite[\textsc{App}~A]{SandersEtAl2007}.

Nowadays these methods represent widely used tools in science and engineering, backed by an increasingly well developed mathematical theory; cf. e.g.~\cite{SandersEtAl2007}. The core idea is to approximate a perturbed full system by an unperturbed averaged system. The latter is thereby obtained from the former by an integration over a scale associated with the perturbation; e.g. a period in the case of a periodic perturbation. The basic theorems then quantify the accuracy of the averaged approximation in terms of error estimates; cf. e.g.~\cite[p.~31, \textsc{Thm}~2.8.1 or p.~101, \textsc{Thm}~5.5.1]{SandersEtAl2007}. Typically the perturbation is controlled by a constant parameter, and the error estimates together with their range of validity are expressed in terms of orders of it.

With respect to modern gravitational physics, averaging often refers to spatial averaging in inhomogenous cosmology, a subject studied in particular with regards to the backreaction source terms appearing in the averaged equations, which potentially provide a natural explanation for dark energy; cf. e.g. \cite{BuchertEhlers1997, Buchert2008}. Here we however consider \emph{spatially homoegenous} cosmology and \emph{temporal} averaging over oscillatory perturbations. The latter occur for instance when the spacetime is populated by matter which satisfies a wave equation, such as scalar fields (Klein-Gordon equation) or Yang-Mills fields; cf. e.g. \cite{FaimanWyatt2021, FajmanEtAl2020, AlhoUggla2015, AlhoEtAl2015, AlhoEtAl2020}. Yet such oscillatory behaviour can even be present in perfect-fluid cosmologies, as for instance demonstrated by the cases which exhibit a phenomenon called Weyl curvature dominance or asymptotic self-similarity breaking; cf. \cite{Wainwright1999, Wainwright2000, Horwood2003}.

\subsection{Motivation}

A focus point of spatially homogenous cosmology is the investigation of the past and future asymptotics of these systems by analytical means, or, expressed in the language of the analysis of PDEs, their stability; cf. \cite{WainwrightEllis1997, Rendall2008, Coley2003, MR3186493, RyanShepley1975,PhysRevD.85.083511, MR3549249, MR3907044,MR2425407} for general reference and also \cite{AnderssonFajman, MR868737} for the stability of some prominent spacetimes.  In other words, one is interested in how the cosmological models evolved in a distant past (e.g. towards a big-bang singularity) or distant future (e.g. towards a big-crunch singularity or in a forever expanding state).  The present work is motivated by~\cite{FajmanEtAl2020} which offers an analysis of the future asymptotics of LRS Bianchi~III cosmologies with a massive scalar field. Type III thereby refers to the Bianchi type of the spatial homogeneity while LRS (locally rotationally symmetric) refers to an additional spatial rotation symmetry; cf.~\cite{RyanShepley1975, WainwrightEllis1997}. The scalar field $\phi=\phi(t)$ satisfies the Klein-Gordon equation
\begin{align}\label{E: KG}
\ddot\phi+\phi=H(-3\dot\phi) ,
\end{align}
in which the mass is set to $1$ and where a dot denotes differentiation with respect to coordinate time $t$. $H=H(t)$ is the Hubble scalar---a measure of the overall isotropic rate of spatial expansion; cf.~\cite{Rendall2008, WainwrightEllis1997}. Evidently, this cosmology is an example for one of the systems mentioned above, in which oscillations enter via a matter wave equation.

From the point of view of~\eqref{E: KG}, $H$ represents a time-dependent damping of the harmonically oscillating matter field $\phi$. Its time dependence is governed by the coupling to the Einstein equations, which can be expressed as a set of constraints and first-order evolution equations for $H$ and other geometric quantities.\footnote{In the present example of LRS Bianchi type~III the only constraint is the Hamiltonian constraint and the only evolution equation in addition to the Raychaudhuri equation is one for the (Hubble normalised) shear; cf.~\cite{FajmanEtAl2020}. The latter has only one independent component in this case.} From the explicit form of the Raychaudhuri equation (i.e. the evolution equation for $H$), it is easy to show that (for an initially expanding universe) $H$ is strictly decreasing and furthermore
\begin{align}\label{E: H to zero}
\lim_{t\to\infty} H(t) = 0;
\end{align}
cf. \cite[\textsc{Lemma}~3]{FajmanEtAl2020}.\footnote{Another important feature of the problem is, that the Raychaudhuri equation is proportional to $H^2$, while the other evolution equations are proportional to $H$.} This guaranteed smallness of $H$ for large times was motivation in~\cite{FajmanEtAl2020} to associate~\eqref{E: KG} with the simpler problem
\begin{align}\label{E: VdP1}
\ddot\phi+\phi=\epsilon\,(-3\dot\phi)
\end{align}
in which $H(t)$ is replaced by a small perturbation constant $\epsilon>0$. For problems of this kind the theory of (periodic) averaging in non-linear dynamical systems offers powerful theorems; cf.~\cite{SandersEtAl2007}. For~\eqref{E: VdP1} these provide error estimates between $\phi$ and the solution of an associated averaged equation in terms of orders of the perturbation constant~$\epsilon$. In the light of~\eqref{E: H to zero} and the association of~\eqref{E: KG} with~\eqref{E: VdP1} it then seems natural to conjecture that the error between the solution of~\eqref{E: KG} and its associated time averaged equation decays and goes to zero. If this holds, then one can obtain the future asymptotics of the full oscillatory system by analysing the much simpler averaged system, which is free of oscillations.

The above paragraph summarises the core idea entertained in~\cite{FajmanEtAl2020}, in which a respective conjecture is formulated and supported. A closely related approach has also been taken independently in~\cite{AlhoUggla2015, AlhoEtAl2015, AlhoEtAl2020}, where \cite{AlhoEtAl2015, AlhoEtAl2020} even provide theorems concerning the future asymptotics. However, the applicability of these theorems is limited to the concrete cosmological models and matter models discussed in these sources, as is also the conjecture of~\cite{FajmanEtAl2020}. Yet oscillatory problems of this kind are a reoccurring theme in mathematical cosmology, at least (but not only) w.~r.~t. matter fields which satisfy a wave equation. Thus, instead of having to prove a separate theorem for each symmetry class and matter model, one would ideally like to have only a short number of theorems available which are however applicable to a wider range of models. With the present work we aim to make a contribution in this spirit. With \textsc{Theorem}~\ref{theorem local} we provide a rigorous result concerning the large-time behaviour of a wide class of cosmological models abstracted from the system of~\cite{FajmanEtAl2020}, i.e. from the LRS Bianchi~III Einstein-Klein-Gordon system. The statement of the theorem is stronger than that of the theorems of \cite{AlhoEtAl2015, AlhoEtAl2020} and the conjecture of~\cite{FajmanEtAl2020} w. r. t. its generality and range of applicability. It is weaker in that it does not provide the solutions in the limit for $t$ to infinity, but rather gives error estimates for finite time intervals. However these time intervals increase, and the estimated errors decrease, with the time of truncation from the full to the averaged system. Since the latter can be chosen freely, our \textsc{Theorem}~\ref{theorem local} should be a versatile and useful tool to investigate the future asymptotics of cosmologies with oscillatory behaviour by analytical means.

Lastly, we note that while we do not explicitly exclude a cosmological constant for the class of models which satisfy the premises of our theorem, it is likely that the requirement that $H\rightarrow 0$ does not hold in these cases. The importance of this decay of $H$ for controlling the error between the full and the averaged solutions is also why we restrict to the future asymptotics here, since towards the past, $H$ and thus the error is increasing in the motivating model of~\cite{FajmanEtAl2020}.

\subsection{Outline}

The paper is structured as follows. In \textsc{Section}~\ref{S: background} we give some background on periodic averaging in non-linear dynamical systems, and we relate our case to the standard problems of this theory. We then formulate our result in \textsc{Section}~\ref{S: Results}: a theorem on the large-time behaviour of a general class of systems, tailored to problems in spatially homogenous cosmology with oscillatory behaviour. We provide the proof in \textsc{Section}~\ref{S: Proofs} and conclude with a summary in~\textsc{Section}~\ref{S: Summary}.

\section{Background and preparation}\label{S: background}

After a word on notation in \textsc{Section}~\ref{SS: Notation} we give some background on averaging theory. In \textsc{Section}~\ref{SS: VdP} we briefly discuss the Van der Pol equations as prototypes for periodic averaging which are closely associated to the Klein-Gordon equation~\eqref{E: KG} of our motivating problem. We then give the standard form of systems in periodic averaging. After a brief discussion of averaging methods and theorems for these systems in \textsc{Section}~\ref{SS: Methods}, we then make the association with our problems of interest in mathematical cosmology, in which we deal with a strictly decreasing perturbation function rather than a perturbation constant. For these we conjecture the existence of respective averaging statements. To this end, we define the cornerstones for our considered class of systems in \textsc{Section}~\ref{SS: class of systems}. In particular we define a required standard form and asymptotics of the perturbation function for our wide class of systems of consideration, but also mild restrictions on the right hand side functions of the systems.

\subsection{Notation}\label{SS: Notation}
Throughout the paper bold case quantities denote vectors whereas quantities in italics denote scalars. We denote the $0^{th}$ to $i^{th}$-order coefficients and the coefficient of the $i^{th}$-order remainder term in the Taylor expansion of $f(\textbf{x},t,\epsilon)$ near $\epsilon=0$ as $f^0(\textbf{x},t),...,f^i(\textbf{x},t)$ and $f^{[i+1]}(\textbf{x},t,\epsilon)$, respectively, such that
$$
f(\textbf{x},t,\epsilon)=\sum_{j=0}^i f^j(\textbf{x},t)\epsilon^j+f^{[i+1]}(\textbf{x},t,\epsilon)\epsilon^{i+1};
$$
cf. also~\cite[\textsc{p}~13, \textsc{Notation}~1.5.2]{SandersEtAl2007}. The norm $\|\cdot\|$ denotes the standard discrete $\ell^1$ norm $$\|\mathbf{u}\|\eqdef \sum_{i=1}^n|u_i|$$ for $\mathbf{u}\in \mathbb{R}^n$. We also denote as $L^\infty_{\textbf{x},t}$ (or simply $L^\infty$) the standard $L^\infty$ space in both $\textbf{x}$ and $t$ variables with the corresponding norm defined as
$$
\|f\|_{L^\infty_{\textbf{x},t}} \eqdef \sup_{\textbf{x},t}|f(\textbf{x},t)|.
$$

\subsection{Van der Pol equations as prototypes of periodic averaging}\label{SS: VdP}

For our purpose it is instructive to compare~\eqref{E: KG} to a general class of Van der Pol equations
\begin{align}\label{E: VdP2}
\ddot\phi+\phi=\epsilon\,g(\phi,\dot\phi) ,
\end{align}
with a sufficiently smooth function $g$ and where $\epsilon$ is a constant, in contrast to $H=H(t)$ in~\eqref{E: KG}; cf.~\cite[\textsc{Sec}~2.2]{SandersEtAl2007}. For small $\epsilon$, \eqref{E: VdP2} represents a perturbed harmonic oscillator which suggests a parametrisation of the solution by variation of arbitrary constants in terms of an amplitude $r=r(t)$ and phase $\varphi=\varphi(t)$:
\begin{align}\label{E: trafo}
\phi = r \sin(t-\varphi) \qquad\text{and}\qquad
\dot\phi = r \cos(t-\varphi) .
\end{align}
This transforms~\eqref{E: VdP2} into two first-order equations of the form $[\dot r,\dot\varphi]^\mathrm T = \epsilon\,\mathbf f(r,\varphi,t)$, with $\mathbf f$ $2\pi$-periodic in $t$. From~\cite[\textsc{Chap}~2]{SandersEtAl2007} this is a periodic averaging problem of the more general standard form
\begin{align}\label{E: standard form}
\dot{\mathbf x} = \epsilon\,\mathbf f^1(\mathbf x, t) + \epsilon^2\, \mathbf f^{[2]}(\mathbf x, t,\epsilon) , \quad
\mathbf x(0)=\mathbf a ,
\end{align}
with $\mathbf f$ $T$-periodic in $t$.

\subsection{Averaging methods and conjecture}\label{SS: Methods}

For this kind of problem the theory of averaging provides theorems which estimate the error between the solution $\mathbf x(t)$ of~\eqref{E: standard form} and a $\mathbf z(t)$ which solves the corresponding averaged equation
\begin{align}\notag
\dot{\mathbf z} = \epsilon\,\overline{\mathbf f}^1(\mathbf z), \quad
\mathbf z(0)=\mathbf a, \quad\text{where}\quad
\overline{\mathbf f}^1(\mathbf z) := \frac{1}{T}\int_0^T \mathbf f^1(\mathbf z, s)\,\mathbf ds .
\end{align}
For instance \cite[\textsc{p}~31, \textsc{Thm}~2.8.1]{SandersEtAl2007} states that the error between $\mathbf x(t)$ and $\mathbf z(t)$ is $O(\epsilon)$ on timescales of $O(\epsilon^{-1})$; cf. also \textsc{Lemma}~\ref{yzmain} below. Note that both the error estimate and its domain of validity are given in terms of orders of powers of the perturbation constant $\epsilon$.
\begin{remark}
The averaging can be interpreted as a time dependent coordinate
transform, a nonlinear transform for nonlinear problems, and many of the results apply to
deterministic time-dependent perturbations; cf. \cite[Chapter 3]{MR3244318} and \cite[Section 5]{MR2585314}.
\end{remark}

The similarity between~\eqref{E: KG} and~\eqref{E: VdP2} suggests to treat the former as a perturbed harmonic oscillator as well, and to apply averaging in an analogous way. Care has to be taken however since in contrast to $\epsilon$, $H$ is time-dependent and itself subject to an evolution equation of the coupled system. If valid, then a striking feature of such an approach is the possibility to exploit the fact~\eqref{E: H to zero} that $H$ is strictly decreasing and going to zero. In~\eqref{E: KG} $H$ takes the role of the perturbation parameter $\epsilon$ in~\eqref{E: VdP2}, which controls the error and the domain of validity of the error estimate. Hence, with strictly decreasing $H$ the error should decrease as well. One can obtain the information about the large-time behaviour of the more complicated full system via an analysis of the simpler averaged system.

\subsection{Our general class of systems}\label{SS: class of systems}

This ideas just summarised were already entertained in~\cite{FajmanEtAl2020} at the example of LRS Bianchi~III. Likewise~\cite{AlhoEtAl2015, AlhoUggla2015, AlhoEtAl2020} took a closely related approach for the concrete models of these sources. \cite{AlhoEtAl2015, AlhoEtAl2020} thereby give theorems concerning the respective future asymptotics.

The purpose of the present work is to formulate a theorem concerning the large-time behaviour which is applicable to a wide range of problems in spatially homogeneous cosmology. We therefore define our problem as general as possible while also as restrictive as necessary to exploit certain key features.

The systems that we consider in the following have the general form
\begin{equation}\label{ODEsystem}
\begin{bmatrix}
\dot{H}\\ \dot{\bfx}
\end{bmatrix}
=H\bfF^1(\bfx,t)+H^2\bfF^{[2]}(\bfx, t)=H
\begin{bmatrix}
0\\ \bff^1(\bfx,t)
\end{bmatrix}+H^2
\begin{bmatrix}
f^{[2]}(\bfx,t)\\ 0
\end{bmatrix},
\end{equation}
where $H=H(t)$ is positive, is strictly decreasing in $t$, and
\begin{equation}\label{E: H to zero 2}
\lim_{t\rightarrow \infty}H(t)=0.
\end{equation}
\begin{example}\label{Ex: 1}
Our prime example for a system which fits the general form \eqref{ODEsystem} is the LRS Bianchi type~III Einstein-Klein-Gordon system introduced in \cite{FajmanEtAl2020}, which is given by 
\begin{align}
\dot{H}&=H^2[-(1+q)], \label{E: BIII H}\\
\dot{\Sigma}_+&=H[-(2-q)\Sigma_++1-\Sigma^2_+-\Omega], \label{E: BIII S}\\
\ddot{\phi}+\phi&=H[-3\dot{\phi}], \label{E: BIII KG}
\end{align}
with the \textit{deceleration parameter} $$q=2\Sigma^2_++\frac{1}{6H^2}(2\dot{\phi}^2-\phi^2).$$
Here $H$ is the \textit{Hubble scalar}, and we see that in this example the first component of~\eqref{ODEsystem} is the the Raychaudhuri equation~\eqref{E: BIII H}. $\phi$ is the scalar field satisfying the Klein-Gordon equation~\eqref{E: BIII KG}. $\Sigma_+$ is the only independent component of the Hubble normalized shear tensor, and $\Omega$ is the Hubble normalized energy density. The \textit{Hamiltonian constraint} $1-\Sigma_+^2-\Omega>0$ and the \textit{non-vacuum} condition $\Omega>0$ give the uniform bounds of $\Sigma_+\in (-1,1)$ and $\Omega \in (0,1)$. The Klein-Gordon equation~\eqref{E: BIII KG} has the form of a Van der Pol equation. Via an amplitude-phase transformation~\eqref{E: trafo} it can be formulated as the following two first-order equations \cite[Section 2.2]{FajmanEtAl2020}:
\begin{equation}\label{KG eq}
\begin{split}
\dot{\Omega}&=H[2\Omega(1+q-3\cos(t-\varphi)^2],\\
\dot{\varphi}&=H[-3\sin(t-\varphi)\cos(t-\varphi)],
\end{split}
\end{equation}which also implies $$q=2\Sigma_+^2+\Omega(3\cos(t-\varphi)^2-1).$$  Then $H$ and $\bfx \eqdef (\Sigma_+, \Omega, \varphi)^\top$ satisfy the system \eqref{ODEsystem}; i.e., the components of $\dot{\mathbf x}$ stand for both the remaining Einstein evolution equation \eqref{E: BIII S} and the two first-order Klein-Gordon equations \eqref{KG eq}. Note that $\bff^1$ and $f^{[2]}$ are smooth with respect to $\bfx$. Also, both $\bff^1$ and $f^{[2]}$ are in $L^\infty_{\bfx, t}.$
\end{example}
\begin{example}\label{Ex: 2}
In addition to LRS Bianchi~III also the Einstein Klein-Gordon systems of the spatially homogenous, LRS cosmologies of Bianchi types~I, II, VIII and IX as well as Kantowski-Sachs can be brought into the form~\eqref{ODEsystem}. The reader can verify that by starting from the formulation of~\cite{RendallUggla2000}, however using a Klein-Gordon field as matter model, and then performing the usual amplitude-phase transformation~\eqref{E: trafo}. For type~IX it is however expected that the condition~\eqref{E: H to zero 2} fails, since it is a closed cosmology, and as such it tends to re-collapse; cf.~\cite{WainwrightEllis1997}.
\end{example}
\begin{example}\label{Ex: 3}
The systems treated in~\cite{AlhoUggla2015, AlhoEtAl2015, AlhoEtAl2020} are also examples of the form~\eqref{ODEsystem}. In \cite{AlhoEtAl2020} the matter source is a Yang-Mills field, which demonstrates that averaging techniques in mathematical cosmology are by no means limited to Einstein-Klein-Gordon cosmology.
\end{example}

\section{Results}\label{S: Results}

We are now ready to formulate our theorem.

\subsection{Large-time behaviour}\label{SS: local-in-time}

Consider the system~\eqref{ODEsystem} with positive  and strictly decreasing $H=H(t)$, and with $\lim_{t\rightarrow \infty}H(t)=0$. We then obtain the following large-time behaviour on the system \eqref{ODEsystem}:
\begin{theorem}[Large-time behaviour]\label{theorem local}Suppose that $H=H(t)>0$ is strictly decreasing in $t$, and $$\lim_{t\rightarrow \infty}H(t)=0.$$ Fix any $\varepsilon>0$ with $\varepsilon<H(0)$ and define $t_*>0$ such that $\varepsilon=H(t_*)$.   Suppose that 
$$\|\bff^1\|_{L^\infty_{\bfx,t}},\ \|f^{[2]}\|_{L^\infty_{\bfx,t}} <\infty,$$ and that $\bff^1(\bfx,t)$ is Lipschitz continuous and $f^{[2]}$ is continuous with respect to $\bfx$ for all $t\ge t_*$. Also, assume that $\bff^1$ and $f^{[2]}$ are $T$-periodic for some $T>0$. Then for all $t>t_*$ with $t=t_*+O(H(t_*)^{-\gamma})$ for any given $\gamma\in(0,1)$, we have
\begin{equation}\label{local statement}\bfx(t)-\bfz(t)=O(H(t_*)^{\min\{1,2-2\gamma\}}),\end{equation} where $\bfx$ is the solution of the system \eqref{ODEsystem} with the initial condition $\bfx(0)=\bfx_0$ and
$\bfz(t)$ is the solution of the averaged equation $$\dot{\bfz}=H(t_*)\bar{\bff}^1(\bfz), \text{ for } t>t_*$$ with the initial condition $\bfz(t_*)=\bfx(t_*)$ where the average $\bar{\bff}^1$ is defined as
$$\bar{\bff}^1(\bfz )=\frac{1}{T}\int_{t_*}^{t_*+T} \bff^1(\bfz,s)ds.$$
\end{theorem}
\begin{proof}
Cf. sections~\ref{S: firstorder} and~\ref{S: periodic averaging}.
\end{proof}
\begin{remark}
Example~\ref{Ex: 1} of \textsc{Section}~\ref{SS: class of systems} satisfies the assumptions of \textsc{Theorem}~\ref{theorem local}. We expect the same to hold for examples~\ref{Ex: 2} and~\ref{Ex: 3}, with the exception of LRS Bianchi~IX which as a closed cosmology is expected to re-collapse and thus to fail the assumption~\eqref{E: H to zero 2}; cf.~\cite{WainwrightEllis1997}.
\end{remark}
\begin{remark}We emphasise the following key properties of~\textsc{Theorem}~\ref{theorem local}, which are connected to the choice of truncation time $t_*$ to the first order system:
\begin{itemize}
\item The longer one waits with truncation, the better the estimate gets.
\item Furthermore, the longer one waits with truncation, for the longer is the estimate valid.
\end{itemize}
These are attractive properties for the analysis of the large-time behaviour of such systems. To the best of our knowledge, \textsc{Theorem}~\ref{theorem local} is the first of its kind, which gives the large-time behaviour for a general class of systems, with a time dependent perturbation parameter, and which naturally appear in spatially homogenous cosmology with oscillating matter models.
\end{remark}


\section{Proofs}\label{S: Proofs}
The proof of \textsc{Theorem}~\ref{theorem local} involves the first-order approximation of the system, the error estimates, and the use of periodic averaging. In \textsc{Section}~\ref{S: firstorder}, we introduce the first-order approximation of our system and prove the main ingredient of our theorem---\textsc{Proposition}~\ref{xymain}. In \textsc{Section}~\ref{S: periodic averaging} we apply a standard theorem of periodic averaging which closes the proof. 

\subsection{First-order approximation and error estimates}\label{S: firstorder}
In order to approximate the solution $\bfx(t)$ of the system \eqref{ODEsystem}, we consider the following truncated first-order system; for $t>t_*$ with some given time $t_*$, we consider
\begin{equation}\label{firstordersystem}
\begin{bmatrix}
\dot{\mathcal{H}}\\ \dot{\bfy}
\end{bmatrix}
=\mathcal{H}(t)
\begin{bmatrix}
0\\ \bff^1(\bfy,t)
\end{bmatrix},
\end{equation}
where $\mathcal{H}(t_*)=H(t_*)$ and $\bfy(t_*)=\bfx(t_*)$; i.e., the initial data coincide. Then since $\dot{\mathcal{H}}=0,$ we have 
$\mathcal{H}(t)\equiv H(t_*).$ We hereby denote $H_*\eqdef H(t_*).$
Then we introduce the following proposition on the error estimates for the first-order approximation:
\begin{proposition}\label{xymain}Fix any $t_*>0$.
Suppose that $\bfx$ and $\bfy$ are the solutions of \eqref{ODEsystem} and \eqref{firstordersystem}, respectively, with the initial conditions $\bfx(0)=\bfx_0$ and $\bfy(t_*)=\bfx(t_*)$.
Suppose further that $H=H(t)>0$ is strictly decreasing in $t$, and $$\lim_{t\rightarrow \infty}H(t)=0.$$ In addition, suppose that 
$$\|\bff^1\|_{L^\infty_{\bfx,t}},\ \|f^{[2]}\|_{L^\infty_{\bfx,t}} <\infty,$$ and that $\bff^1(\bfx,t)$ is Lipschitz continuous in $\bfx$ for all $t\ge t_*$. Then for all $t>t_*$ with $t=t_*+O(H(t_*)^{-\gamma})$ for some $\gamma\in(0,1)$, we have
$$\bfx(t)-\bfy(t)=O(H(t_*)^{2-2\gamma}).$$ 
\end{proposition}
\begin{proof}
Since $\dot{\mathcal{H}}=0$ in \eqref{firstordersystem}, we have $\mathcal{H}(t)=H_*,$ for all $t\ge t_*.$ Now, we subtract the equation for $\bfy$ in \eqref{firstordersystem} from that for $\bfx$ in \eqref{ODEsystem} to obtain that
$$\dot{\bfx}-\dot{\bfy} = H(t)\bff^1(\bfx,t)-H_* \bff^1(\bfy,t).$$ By integrating this identity, we have
\begin{equation}\label{eq 1} \bfx(t)-\bfy(t) = \int_{t_*}^t \left(H(\tau)\bff^1(\bfx,\tau)-H_* \bff^1(\bfy,\tau)\right)d\tau,\end{equation} by the initial condition $\bfy(t_*)=\bfx(t_*).$
Now we consider the mean-value theorem for $H(\tau)$ and observe that, for some $t_c=t_c(\tau) \in (t_*,\tau)$, we have 
$$H(\tau)-H_*=\dot{H}(t_c) (\tau-t_*),\ \text{ for } \tau \in (t_*,t).$$
Then \eqref{eq 1} implies that
\begin{multline*}\bfx(t)-\bfy(t) = \int_{t_*}^t \left((H_*+\dot{H}(t_c) (\tau-t_*))\bff^1(\bfx,\tau)-H_* \bff^1(\bfy,\tau)\right)d\tau\\
= H_*\int_{t_*}^t \left(\bff^1(\bfx,\tau)-\bff^1(\bfy,\tau)\right)d\tau
+\int_{t_*}^t \dot{H}(t_c) (\tau-t_*)\bff^1(\bfx,\tau)d\tau\\
=H_*\int_{t_*}^t \left(\bff^1(\bfx,\tau)-\bff^1(\bfy,\tau)\right)d\tau
+\int_{t_*}^t H^2(t_c) f^{[2]}(\bfx,t_c) (\tau-t_*)\bff^1(\bfx,\tau)d\tau,
\end{multline*}
where the last line is by the first equation on $\dot{H}$ in the system \eqref{ODEsystem}.
Therefore, we have
\begin{multline*}
\|\bfx(t)-\bfy(t)\| \\
\le H_*\int_{t_*}^t \left\|\bff^1(\bfx,\tau)-\bff^1(\bfy,\tau)\right\|d\tau
+\int_{t_*}^t H^2(t_c) |f^{[2]}(\bfx,t_c) | (\tau-t_*) \|\bff^1(\bfx,\tau)\|d\tau\\
\le H_*\int_{t_*}^t \left\|\bff^1(\bfx,\tau)-\bff^1(\bfy,\tau)\right\|d\tau
+\frac{(t-t_*)^2}{2}H_*^2\|f^{[2]}\|_{L^\infty}\|\bff^1\|_{L^\infty}\\
\le H_* c_L\int_{t_*}^t \left\|\bfx(\tau)-\bfy(\tau)\right\|d\tau
+\frac{(t-t_*)^2}{2}H_*^2\|f^{[2]}\|_{L^\infty}\|\bff^1\|_{L^\infty},
\end{multline*} because $H(\cdot)$ is decreasing, 
where $c_L$ is the Lipschitz constant and the third inequality is by the Lipschitz continuity of $\bff^1.$ Now we apply the Gr\"onwall inequality and obtain 
$$\|\bfx(t)-\bfy(t)\| \le \frac{(t-t_*)^2}{2}H_*^2\|f^{[2]}\|_{L^\infty}\|\bff^1\|_{L^\infty}\exp\left(H_* c_L(t-t_*)\right).$$
Therefore, for the time-scale $t=t_*+O(H_*^{-\gamma})$ for any $\gamma \in(0,1)$ we have
$$\|\bfx(t)-\bfy(t)\| \le O(H_*^{2-2\gamma})\exp(O(H_*^{1-\gamma}))=O(H_*^{2-2\gamma}),$$ as long as $H_*$ is finite. This completes the proof.
\end{proof}

\subsection{Periodic averaging}\label{S: periodic averaging}

Equipped with the first-order approximation and the error estimates, the next ingredient is a result from periodic averaging theory.

First we remind ourselves of the periodic averaging problem in standard form~\eqref{E: standard form} which we discussed in \textsc{Section}~\ref{SS: VdP}. For convenience we rewrite it here for the special case $\mathbf f^{[2]}=0$ as
\begin{equation}\label{yeq}
\dot{\bfy}=\varepsilon\bff^1(\bfy,t),\ \bfy(0)=\bfy_0,\quad \text{for } t>0
\end{equation}
with $\bfy, \ \bff^1(\bfy,t)\in \mathbb{R}^n,$ where $\bff^1$ is $T$-periodic and $\varepsilon$ is a small parameter.
Now we note that the equation for $\bfy$ of our first order system~\eqref{firstordersystem} is of this form with $\varepsilon=H_*$.
Then we fix a sufficiently large time $t_*>0$, and consider the first-order associated averaged system
\begin{equation}\label{zeq}\dot{\bfz}=\varepsilon\bar{\bff}^1(\bfz), \text{ for } t>t_*\end{equation} with the initial condition $\bfz(t_*)=\bfy(t_*)$ where $\varepsilon=H(t_*)$ and the average $\bar{\bff}^1$ is defined as
$$\bar{\bff}^1(\bfz )=\frac{1}{T}\int_{t_*}^{t_*+T} \bff^1(\bfz,s)ds.$$
Let $D\subset \mathbb{R}^n$ be a connected, bounded open set (with compact closure) containing the initial value $\bfy(t_*)=\bfz(t_*)$. Fix $L>0$ and $\varepsilon_0>0$ such that the solutions $\bfy(t,\varepsilon)$ and $\bfz(t,\varepsilon)$ with $0\le \varepsilon\le \varepsilon_0$ remain in $D$ for $t_*\le t\le t_*+L/\varepsilon.$ Then there is a basic result on the asymptotic error estimates:
\begin{lemma}[Theorem 2.8.1 of \cite{SandersEtAl2007}]\label{yzmain}
Let $\bff^1$ be Lipschitz continuous.  Also, assume that $\bff^1$ is $T$-periodic for some $T>0$. Let $\varepsilon_0$, $D$, $L$ be as above. Then there exists a constant $c>0$ such that the solutions $\bfy$ and $\bfz$ of \eqref{yeq} and \eqref{zeq} satisfy
$$\|\bfy(t,\varepsilon)-\bfz(t,\varepsilon)\|<c\varepsilon,$$for $0\le \varepsilon\le \varepsilon_0$ and $t_*\le t \le t_*+\frac{L}{\varepsilon},$ where $\|\cdot\|$ denotes the norm $$\|\mathbf{u}\|\eqdef \sum_{i=1}^n|u_i|$$ for $\mathbf{u}\in \mathbb{R}^n.$
\end{lemma}

Then Proposition \ref{xymain} and Lemma \ref{yzmain} together imply \textsc{Theorem}~\ref{theorem local} by letting $\varepsilon=H(t_*)$, using the triangle inequality and noting that the time scale $O(H(t_*)^{-\gamma})$ is smaller than $O(H(t_*)^{-1})$ for a sufficiently large time $t_*$ such that $H(t_*)=\varepsilon\ll1.$ 
This completes the proof of \textsc{Theorem}~\ref{theorem local}.

\section{Summary}\label{S: Summary}

Motivated by problems in spatially homogenous cosmology in \cite{FajmanEtAl2020, AlhoEtAl2015, AlhoUggla2015, AlhoEtAl2020} in which oscillations enter the system via a matter wave equation, we present a theorem regarding the large-time behaviour of such systems. The key idea is thereby to view these as perturbation problems, in which the Hubble scalar plays the role of the perturbation function---in analogy to the perturbation constant in standard problems of averaging in non-linear dynamical systems.

In standard averaging the perturbation constant estimates the error between the solutions of the full and a simpler time-averaged system. We show that for large times this is also the case for our class of systems and the Hubble scalar. Furthermore, we show that the estimated error decreases, and the time of validity of the estimate increases, with the time of truncation from the full to the averaged system, which can be chosen freely.

While related theorems have been given for the concrete examples discussed in~\cite{AlhoEtAl2015, AlhoEtAl2020}, our theorem is formulated in a very general fashion, with only mild restrictions on the systems of consideration. Since systems of this kind appear to be re-occurring in spatially homogenous cosmology with oscillatory behaviour (cf. \cite{FajmanEtAl2020, AlhoEtAl2015, AlhoUggla2015, AlhoEtAl2020}) it is desirable to craft a toolkit of theorems which is applicable to a wider class of models. To this end we offer our theorem, which to our best knowledge represents the first step in this direction.



\bibliographystyle{hplain}

\bibliography{BibFile}{}
\end{document}